\documentclass[english]{amsart}
\usepackage[T1]{fontenc}
\usepackage[latin9]{inputenc}
\usepackage[active]{srcltx}
\usepackage{mathtools}
\usepackage{mathrsfs}
\usepackage{amstext}
\usepackage{amsthm}
\usepackage{amssymb}
\usepackage[colorlinks, linkcolor=blue, anchorcolor=blue, citecolor=blue]{hyperref}

\makeatletter
\numberwithin{equation}{section}
\numberwithin{figure}{section}
\theoremstyle{plain}
\newtheorem{thm}{\protect\theoremname}
  \theoremstyle{plain}
  \newtheorem{lem}[thm]{\protect\lemmaname}
  \theoremstyle{plain}
  \newtheorem{prop}[thm]{\protect\propositionname}

\usepackage{a4wide}
\usepackage{comment}

\usepackage{babel}

  \providecommand{\lemmaname}{Lemma}
  \providecommand{\propositionname}{Proposition}
\providecommand{\theoremname}{Theorem}

\makeatother

\usepackage{babel}
  \providecommand{\lemmaname}{Lemma}
  \providecommand{\propositionname}{Proposition}
\providecommand{\theoremname}{Theorem}

\begin{document}

\title{Remarks on factoriality and $q$-deformations}

\author{Adam Skalski}
\address{Institute of Mathematics of the Polish Academy of Sciences,
	ul.~\'Sniadeckich 8, 00--656 Warszawa, Poland}
\email{a.skalski@impan.pl}

\author{Simeng Wang}

\address{Laboratoire de Mathématiques, Université de Franche-Comté, 25030
	Besançon Cedex, France and Institute of Mathematics of the Polish Academy of Sciences,
	ul.~\'Sniadeckich 8, 00--956 Warszawa, Poland}

\email{simeng.wang@univ-fcomte.fr}
\begin{abstract}
We prove that the mixed $q$-Gaussian algebra $\Gamma_{Q}(H_{\mathbb{R}})$
associated to a real Hilbert space $H_{\mathbb{R}}$ and a real symmetric
matrix $Q=(q_{ij})$ with $\sup|q_{ij}|<1$, is a factor as soon as
$\dim H_{\mathbb{R}}\geq2$. We also discuss the factoriality of $q$-deformed
Araki-Woods algebras, in particular showing that the $q$-deformed
Araki-Woods algebra $\Gamma_{q}(H_{\mathbb{R}},U_{t})$ given by a
real Hilbert space $H_{\mathbb{R}}$ and a strongly continuous group $U_{t}$ is a factor
when $\dim H_{\mathbb{R}}\geq2$ and $U_{t}$ admits an invariant
eigenvector.
\end{abstract}

\maketitle

\section{Introduction}

This paper studies the factoriality of some $q$-deformed von Neumann
algebras. In the early 90's, motivated by mathematical physics, Bo\.{z}ejko
and Speicher introduced the von Neumann algebra $\Gamma_{q}(H_{\mathbb{R}})$
generated by $q$-Gaussian variables \cite{bozejkospeicher91q}. Since
then, the von Neumann algebra $\Gamma_{q}(H_{\mathbb{R}})$ has been widely
studied, and also its several generalizations have been introduced and fruitfully investigated.
In particular, there are two interesting types of $q$-deformed algebras
which generalize that of Bo\.{z}ejko and Speicher: the first one is
the mixed $q$-Gaussian algebra introduced in \cite{bozejkospeicher94yangbaxter},
and the second one is the family of $q$-deformed Araki-Woods algebras constructed
in \cite{hiai03qaw}.

The question of factoriality of these $q$-deformed Neumann algebras remained
a well-known problem in the field for many years. In 2005, Ricard
\cite{ricard05qfactor} proved that the von Neumann algebra $\Gamma_{q}(H_{\mathbb{R}})$
is a factor as soon as $\dim H_{\mathbb{R}}\geq2$, which solved the
problem for $\Gamma_{q}(H_{\mathbb{R}})$ in full generality (for earlier partial results see
also \cite{sniady04qfactor}, \cite{krolak06factoriality}, \cite{bozejkokummerspeicher97q}). However,
the analogous problem for mixed $q$-Gaussian algebras and $q$-deformed
Araki-Woods algebras has remained open. 
Among the known results, the factoriality of mixed $q$-Gaussian algebras was proved by Kr\'olak \cite{krolak00wickyangbaxter} when the underlying
Hilbert space is infinite-dimensional, and very recently by Nelson and Zeng \cite{nelsonzeng16mixedq}
when the size of the deformation parameters is sufficiently small; similarly, the factoriality of
$q$-deformed Araki-Woods algebras was only established by Hiai in \cite{hiai03qaw}
when the `almost periodic part' (see Section 4 for an explanation of this term) of the underlying Hilbert space is infinite-dimensional, and by Nelson in
\cite{nelson15freenontracialtransport} when $q$ is small.

In this note we solve the problem of factoriality for mixed $q$-Gaussian
algebras in full generality, following the ideas of \cite{ricard05qfactor}. Our methods apply also to the $q$-deformed Araki-Woods
algebras, and we show that the $q$-deformed Araki-Woods algebra
$\Gamma_{q}(H_{\mathbb{R}},U_{t})$ is a factor as soon as $\dim H_{\mathbb{R}}\geq2$
and the semigroup $U_{t}$ admits an invariant eigenvector. 
We remark that after the completion of this work, we learned that the
last result mentioned above was also obtained
independently by Bikram and Mukherjee in \cite{bikrammukherjee16qawfactor}, as a part of a detailed study of maximal abelian subalgebras in $q$-deformed Araki-Woods algebras.

The scalar products below are always linear on the left. The plan of the paper is as follows: in  Section 2 we present a Hilbert space lemma providing estimates for certain commutators to be used later, in Section 3 we establish the factoriality of mixed $q$-Gaussian algebras in full generality, and in Section 4 discuss several results concerning factoriality in the context of $q$-Araki-Woods von Neumann algebras.

\section{A convergence lemma for $q$-commutation relations}

The following purely Hilbert-space-theoretic lemma will play a key role in our discussions of factoriality in the following sections.
\begin{lem}
\label{lem:conv general case}Let $(H_{n})_{n\geq1}$ be a sequence
of Hilbert spaces and write $H=\oplus_{n\geq1}H_{n}$. Let $r,s \in \mathbb{N}$ and let  $(a_{i})_{1\leq i\leq r}$, $(b_{j})_{1\leq j\leq s}$
be two families of operators on $H$ which send each $H_{n}$ into
$H_{n+1}$ or $H_{n-1}$, such that there exists $0<q<1$ with
\[
\|(a_{i}b_{j}-b_{j}a_{i})|_{H_{n}}\|\leq q^{n}, \;\;\; n \in \mathbb{N}.
\]
Assume that $K_{n}\subset H_{n}$ is a finite-dimensional Hilbert
subspace for each $n\geq1$ such that for $K=\oplus_{n}K_{n}$ we have
\[
a_{i}(K)\subset K,\quad1\leq i\leq r-1,\quad\text{and }a_{r}|_{K}=0.
\]
 Then for any bounded nets $(\xi_{\alpha}),(\eta_{\alpha})\subset K$
such that $\eta_{\alpha}\to0$ weakly, we have
\[
\langle a_{1}^{*}\cdots a_{r}^{*}\xi_{\alpha},b_{1}\cdots b_{s}\eta_{\alpha}\rangle\to0.
\]
\end{lem}
\begin{proof}
Put
\[
T_{ij}^{(n)}=(a_{i}b_{j}-b_{j}a_{i})|_{H_{n}},\quad1\leq i\leq r,1\leq j\leq s,n\geq1.
\]
Then for each $i$ we may write
\[
a_{i}b_{1}\cdots b_{s}\xi-b_{1}\cdots b_{s}a_{i}\xi=\sum_{j=1}^{s}b_{1}\cdots b_{j-1}T_{ij}^{(m(j,n))}b_{j+1}\cdots b_{s}\xi,\quad\xi\in H_{n},
\]
where $m(j,n)$ is an integer greater than $n-s$. Iterating this formula
we obtain
\begin{align*}
a_{r}\cdots a_{1}b_{1}\cdots b_{s}\xi & =b_{1}\cdots b_{s}a_{r}\cdots a_{1}\xi+\sum_{i=1}^{r}(a_{r}\cdots a_{i}b_{1}\cdots b_{s}a_{i-1}\cdots a_{1}\xi-a_{r}\cdots a_{i+1}b_{1}\cdots b_{s}a_{i}\cdots a_{1}\xi)\\
 & =b_{1}\cdots b_{s}a_{r}\cdots a_{1}\xi+\sum_{i=1}^{r}a_{r}\cdots a_{i+1}\left(\sum_{j=1}^{s}b_{1}\cdots b_{j-1}T_{ij}^{(m'(i,j,n))}b_{j+1}\cdots b_{s}\right) a_{i-1}\cdots a_{1}\xi,
\end{align*}
where $\xi\in H_{n}$ and for each $i,j,n$ the integer $m'(i,j,n)$ is greater that
$n-s-r$. Now we consider two bounded nets $(\xi_{\alpha}),(\eta_{\alpha})\subset K$
such that $\eta_{\alpha}\to0$ weakly. Write
\[
\eta_{\alpha}=(\eta_{\alpha}^{(n)})_{n\geq1},\quad\eta_{\alpha}^{(n)}\in K_{n}.
\]
We have
\[
\langle a_{1}^{*}\cdots a_{r}^{*}\xi_{\alpha},b_{1}\cdots b_{s}\eta_{\alpha}\rangle=\langle\xi_{\alpha},a_{r}\cdots a_{1}b_{1}\cdots b_{s}\eta_{\alpha}\rangle,
\]
and by the assumptions $a_{r}\cdots a_{1}\eta_{\alpha}=0$, so together
with the previous computations for $a_{r}\cdots a_{1}b_{1}\cdots b_{s}\xi$,
we obtain
\begin{equation}
\langle a_{1}^{*}\cdots a_{r}^{*}\xi_{\alpha},b_{1}\cdots b_{s}\eta_{\alpha}\rangle=\sum_{n\geq1}\langle\xi_{\alpha},T_{n}\eta_{\alpha}^{(n)}\rangle,\label{eq:lem conv inner prod}
\end{equation}
where
\[
T_{n}=\sum_{i=1}^{r}a_{r}\cdots a_{i+1}\left(\sum_{j=1}^{s}b_{1}\cdots b_{j-1}T_{ij}^{(m'(i,j,n))}b_{j}\cdots b_{s}\right)a_{i-1}\cdots a_{1}.
\]
Recall that $\|T_{ij}^{(k)}\|\leq q^{k}$ for all $i,j,k$ by assumption.
So for each $\alpha$ and $n$
\[
\|T_{n}\eta_{\alpha}^{(n)}\|\leq C(q,r,s)q^{n}\|\eta_{\alpha}^{(n)}\|,
\]
where $C(q,r,s)$ is a constant independent of $n$. Together with
\eqref{eq:lem conv inner prod} we have
\begin{align}
|\langle a_{1}^{*}\cdots a_{r}^{*}\xi_{\alpha},b_{1}\cdots b_{s}\eta_{\alpha}\rangle| & \leq C(q, r,s)\sup_{\alpha}\|\xi_{\alpha}\|\sum_{n\geq1}q^{n}\|\eta_{\alpha}^{(n)}\|.\label{eq:lem conv estimate for inner prod}
\end{align}
Since $\eta_{\alpha}\to0$ weakly, we have for each $N\geq1$,
\[
\sum_{n=1}^{N}q^{n}\|\eta_{\alpha}^{(n)}\|\xrightarrow[\alpha]{}0,
\]
and on the other hand,
\[
\sum_{n\geq N}q^{n}\|\eta_{\alpha}^{(n)}\|\leq\sup_{n}\|\eta_{\alpha}^{(n)}\|q^{N}/(1-q).
\]
Therefore by \eqref{eq:lem conv estimate for inner prod} we get
\[
\forall N\geq1,\quad\limsup_{\alpha}|\langle a_{1}^{*}\cdots a_{r}^{*}\xi_{\alpha},b_{1}\cdots b_{s}\eta_{\alpha}\rangle|\leq C'(r,s,q)q^{N},
\]
with a constant $C'(r,s,q)$ independent of $N$, which means that
\[
\langle a_{1}^{*}\cdots a_{r}^{*}\xi_{\alpha},b_{1}\cdots b_{s}\eta_{\alpha}\rangle\to0,
\]
as desired.
\end{proof}

\section{Factoriality of mixed $q$-Gaussian algebras}

Let $N \in \mathbb{N}$, let $Q=(q_{ij})_{i,j=1}^N$ be a symmetric matrix with $q_{ij}\in(-1,1)$,
and let $H_{\mathbb{R}}$ be a finite-dimensional real Hilbert space with
orthonormal basis $e_{1},\ldots,e_{N}$. We recall briefly the construction
of mixed Gaussian algebras, as introduced in \cite{bozejkospeicher94yangbaxter}.
Write $H=H_{\mathbb{R}}+\mathrm{i}H_{\mathbb{R}}$ to be the complexification
of $H_{\mathbb{R}}$. Let $\mathcal{F}_{Q}(H)$ be the Fock space
associated to the Yang-Baxter operator
\[
T:H\otimes H\to H\otimes H,\quad e_{i}\otimes e_{j}\mapsto q_{ij}e_{j}\otimes e_{i}
\]
constructed in \cite{bozejkospeicher94yangbaxter}. Denote by $\langle\cdot,\cdot\rangle$
the inner product on $\mathcal{F}_{Q}(H)$ and let $\Omega$ be the
vacuum vector. Denote by $\varphi(\cdot)=\langle\cdot\Omega,\Omega\rangle$
the vacuum state. The left creation operators $l_{i}$ are defined
by the formulas
\[
l_{i}\xi=e_{i}\otimes\xi,\quad\xi\in\mathcal{F}_{Q}(H),
\]
and their adjoints, the left annihilation operators, can be characterised  by equalities
\[
l_{i}^{*}\Omega=0,
\]
\[
l_{i}^{*}(e_{j_{1}}\otimes\cdots\otimes e_{j_{n}})=\sum_{k=1}^{n}\delta_{i,j_{k}}q_{ij_{1}}\cdots q_{ij_{k-1}}e_{j_{1}}\otimes\cdots\otimes e_{j_{k-1}}\otimes e_{j_{k+1}}\otimes\cdots\otimes e_{j_{n}}.
\]
Similarly, we have the right creation/annihilation operators
\[
r_{i}\xi=\xi\otimes e_{i},\quad\xi\in\mathcal{F}_{Q}(H),
\]
\[
r_{i}^{*}\Omega=0,
\]
\[
r_{i}^{*}(e_{j_{1}}\otimes\cdots\otimes e_{j_{n}})=\sum_{k=1}^{n}\delta_{i,j_{k}}q_{ij_{k+1}}\cdots q_{ij_{n}}e_{j_{1}}\otimes\cdots\otimes e_{j_{k-1}}\otimes e_{j_{k+1}}\otimes\cdots\otimes e_{j_{n}}.
\]
We consider the associated mixed $q$-Gaussian algebra $\Gamma_{Q}(H_{\mathbb{R}})$
generated by the self-adjoint variables $s_{j}=l_{j}^{*}+l_{j}$.
Denote
\[
q=\max_{i,j}|q_{ij}|<1.
\]
By a \emph{word} in $\mathcal{F}_{Q}(H)$ we mean a vector in $\mathcal{F}_{Q}(H)$ of the
form $\zeta_{1}\otimes\cdots\otimes \zeta_{n}$ with some $n\geq1$ and $\zeta_{1},\ldots,\zeta_{n}\in H$.
Kr\'olak \cite{krolak00wickyangbaxter} proved that any word $\xi\in\mathcal{F}_{Q}(H)$
corresponds to a \emph{Wick product} $W(\xi)\in\Gamma_{Q}(H_{\mathbb{R}})$
with $W(\xi)\Omega=\xi$. Also, \cite{bozejkospeicher94yangbaxter}
remarked that $J\Gamma_{Q}(H_{\mathbb{R}})J$ is the commutant of
$\Gamma_{Q}(H_{\mathbb{R}})$, where $J$ is the conjugation operator
given by
\[
J(e_{i_{1}}\otimes\cdots\otimes e_{i_{n}})=e_{i_{n}}\otimes\cdots\otimes e_{i_{1}}.
\]
We write
\[
W_{r}(\xi)=JW(J\xi)J,\quad\xi\in\oplus_{n}H^{\otimes n}.
\]
Then $W_{r}(\xi)\in\Gamma_{Q}(H_{\mathbb{R}})'$.
\begin{lem}
\label{lem:commutator mixed}For each $n \in \mathbb{N}$ and $i,j=1,\ldots,N$ the operators $T_{i}^{(n)}$
on $H^{\otimes n}$  characterised by the equalities
\[
l_{i}^{*}r_{j}-r_{j}l_{i}^{*}=\delta_{ij}\oplus_{n}T_{i}^{(n)}.
\]
satisfy the norm estimate $\|T_{i}^{(n)}\|\leq q^{n}$.
\end{lem}
\begin{proof}
The case of $n=0$ is obvious and we take $n\geq1$ in the following.
Observe that
\begin{align*}
l_{i}^{*}r_{j}(e_{j_{1}}\otimes\cdots\otimes e_{j_{n}}) & =l_{i}^{*}(e_{j_{1}}\otimes\cdots\otimes e_{j_{n}}\otimes e_{j})\\
 & =\sum_{k=1}^{n}\delta_{i,j_{k}}q_{ij_{1}}\cdots q_{ij_{k-1}}e_{j_{1}}\otimes\cdots\otimes e_{j_{k-1}}\otimes e_{j_{k+1}}\otimes\cdots\otimes e_{j_{n}}\otimes e_{j}\\
 & \quad\ +\delta_{ij}q_{ij_{1}}\cdots q_{ij_{n}}e_{j_{1}}\otimes\cdots\otimes e_{j_{n}},
\end{align*}
and
\[
r_{j}l_{i}^{*}(e_{j_{1}}\otimes\cdots\otimes e_{j_{n}})=\sum_{k=1}^{n}\delta_{i,j_{k}}q_{ij_{1}}\cdots q_{ij_{k-1}}e_{j_{1}}\otimes\cdots\otimes e_{j_{k-1}}\otimes e_{j_{k+1}}\otimes\cdots\otimes e_{j_{n}}\otimes e_{j}.
\]
Now take
\[
T_{i}^{(n)}:H^{\otimes n}\to H^{\otimes n},\quad e_{j_{1}}\otimes\cdots\otimes e_{j_{n}}\mapsto\delta_{ij}q_{ij_{1}}\cdots q_{ij_{n}}e_{j_{1}}\otimes\cdots\otimes e_{j_{n}}.
\]
The eigenspace of $T_{i}^{(n)}$ corresponding to $\delta_{ij}q_{ij_{1}}\cdots q_{ij_{n}}$
is spanned by the vectors of the type $E_{\{j_{1},\ldots,j_{n}\}}=\{e_{j_{1}'}\otimes\cdots\otimes e_{j_{n}'}:q_{ij_{1}}\cdots q_{ij_{n}}=q_{ij_{1}'}\cdots q_{ij_{n}'}\}$,
which are orthogonal for distinct $\underline{j}=\{j_{1},\ldots,j_{n}\}$.
So
\[
\|T_{i}^{(n)}\|\leq\max\{q_{ij_{1}}\cdots q_{ij_{n}}:1\leq j_{1},\ldots,j_{n}\leq N\}\leq q^{n}
\]
and $T_{i}^{(n)}$ is the desired operator.
\end{proof}
Now the following main result is in reach. The idea is partially inspired
by the proof in \cite{ricard05qfactor} in conjunction with Lemma
\ref{lem:conv general case}.
\begin{thm}
\label{thm:factor mixed}For each $1\leq i\leq n$, the von Neumann
subalgebra generated by $s_{i}$ is maximal abelian in $\Gamma_{Q}(H_{\mathbb{R}})$.
In particular, $\Gamma_{Q}(H_{\mathbb{R}})$ is a factor if $n\geq2$.\end{thm}
\begin{proof}
By \cite{bozejkokummerspeicher97q}, we know that the spectral measure
of $s_{i}$ is the $q$-semicircular law with $q=q_{ii}$. Therefore
the von Neumann algebra $M$ generated by $s_{i}$ is diffuse and
abelian, and hence  isomorphic to the von Neumann algebra $L^{\infty}([0,1],dm)$
where $dm$ denotes the Lebesgue measure on $[0,1]$. As a result,
we may find a sequence of unitaries $(u_{\alpha})_{\alpha\in \mathbb{N}}\subset M$
which correspond to Rademacher functions via this isomorphism. In
particular, we have
\[
u_{\alpha}=u_{\alpha}^{*},\quad u_{\alpha}^{2}=1,\quad u_{\alpha}\Omega\to0\text{ weakly in }\mathcal{F}_{Q}(H).
\]
Now assume $x\in\Gamma_{Q}(H_{\mathbb{R}})$ with $xs_{i}=s_{i}x$,
and hence
\[
xy=yx,\quad y\in M.
\]
Let $\mathcal{F}_{Q}(\mathbb{C}e_{i})\subset\mathcal{F}_{Q}(H)$ be
the Fock space associated to $e_{i}$. Observe that for any vector
$\xi\in \bigcup_{m\in \mathbb{N}} H^{\otimes m}$  and all $\alpha\geq1$ we have
\begin{equation}
\langle\xi,x\Omega\rangle=\varphi(x^{*}W(\xi))=\varphi(x^{*}u_{\alpha}^{2}W(\xi))=\varphi(u_{\alpha}x^{*}u_{\alpha}W(\xi))=\langle W_{r}(\xi)u_{\alpha}\Omega,xu_{\alpha}\Omega\rangle.\label{eq:mixed inner prod zero}
\end{equation}
We remark that if further $\xi$ is orthogonal to $\mathcal{F}_{Q}(\mathbb{C}e_{i})$ then
\begin{equation}
\forall y\in\Gamma_{Q}(H_{\mathbb{R}}),\quad\langle W_{r}(\xi)u_{\alpha}\Omega,yu_{\alpha}\Omega\rangle\to0.\label{eq:mixed conv infinite}
\end{equation}
To see this, it suffices to consider the case $y\Omega\in H^{\otimes n}$
for an arbitrary $n\geq0$ since it is easy to see that the functionals
$y^*\Omega\mapsto\langle W_{r}(\xi)u_{\alpha}\Omega,yu_{\alpha}\Omega\rangle$
extend to uniformly bounded functionals on $\mathcal{F}_{Q}(H)$ thanks
to the traciality of $\varphi$ (\cite[Theorem 4.4]{bozejkospeicher94yangbaxter}).
Now by the Wick formula in \cite[Theorem 1]{krolak00wickyangbaxter},
it is enough to prove the convergence
\begin{equation}
\langle r_{i_{1}}\cdots r_{i_{s}}r_{i_{s+1}}^{*}\cdots r_{i_{p}}^{*}u_{\alpha}\Omega,l_{j_{1}}\cdots l_{j_{t}}l_{j_{t+1}}^{*}\cdots l_{j_{q}}^{*}u_{\alpha}\Omega\rangle\to0\label{eq:them mixed conv}
\end{equation}
for any fixed indices $i_{1},\ldots,i_{p},j_{1},\ldots,j_{q}$ with
some $i_{k}\neq i$. Denote
\[
s'=\min\{k:i_{k}\neq i\}.
\]
If $s'>s$, we have $r_{i_{s'}}^{*}\cdots r_{i_{p}}^{*}u_{\alpha}\Omega=0$
for all $\alpha\geq1$ and the convergence \eqref{eq:them mixed conv}
becomes trivial. So we assume in the following $s'\leq s$. Note that
by definition
\[
r_{i}l_{j}-l_{j}r_{i}=0,\quad r_{i}^{*}l_{j}^{*}-l_{j}^{*}r_{i}^{*}=0,
\]
and by Lemma \ref{lem:commutator mixed}
\[
\|(l_{i}^{*}r_{j}-r_{j}l_{i}^{*})|_{H^{\otimes n}}\|\leq q^{n},\quad n\geq1.
\]
Also, observe that by the choice of $s'$,
\[
r_{i_{s'}}^{*}|_{\mathcal{F}_{Q}(\mathbb{R}e_{i})}=0,\quad r_{i_{k}}^{*}(\mathcal{F}_{Q}(\mathbb{R}e_{i}))\subset\mathcal{F}_{Q}(\mathbb{R}e_{i}),\quad1\leq k<s'.
\]
So now applying Lemma \ref{lem:conv general case} to the families
of operators $r_{i_{1}}^{*},\ldots,r_{i_{s'}}^{*}$ and $l_{j_{1}},\ldots,l_{j_{t}},l_{j_{t+1}}^{*},\ldots,l_{j_{q}}^{*}$,
we obtain the convergence \eqref{eq:them mixed conv}. As a consequence,
the convergence \eqref{eq:mixed conv infinite} holds as well, which,
together with \eqref{eq:mixed inner prod zero}, yields that
\[
\langle\xi,x\Omega\rangle=0.
\]
This means that $x\Omega\in\mathcal{F}_{Q}(\mathbb{C}e_{i})$ since
$\xi$ is arbitrarily chosen in a dense subset of $\mathcal{F}_{Q}(\mathbb{C}e_{i})^{\bot}$.
We can then deduce that $x\in M$ using the second quantization of the projection
$P:H_{\mathbb{R}}\to\mathbb{R}e_{i}$ (see \cite[Lemma 3.1]{lustpiquard99qfock}).
Thus we have shown that the von Neumann subalgebra $M$ generated by $s_{i}$
is maximal abelian in $\Gamma_{Q}(H_{\mathbb{R}})$.

Also, if $x\in\Gamma_{Q}(H_{\mathbb{R}})\cap\Gamma_{Q}(H_{\mathbb{R}})'$,
then the above argument shows that $x\Omega\in\cap_{i=1}^{n}\mathcal{F}_{Q}(\mathbb{C}e_{i})$,
so $x\Omega\in\mathbb{C}\Omega$. Therefore $\Gamma_{Q}(H_{\mathbb{R}})$
is a factor.
\end{proof}

\section{Factoriality of $q$-Araki-Woods algebras}

Now we discuss the factoriality of $q$-Araki-Woods algebras. We refer
to \cite{hiai03qaw} for the detailed description of the construction of these
algebras and only sketch the outline below. Following the notation of \cite{hiai03qaw}, given a real
Hilbert space $H_{\mathbb{R}}$ with a strongly continuous group $U_{t}$
of orthogonal transformations on $H_{\mathbb{R}}$, we may introduce
a deformed inner product $\langle\cdot,\cdot\rangle_{U}$ on $H_{\mathbb{C}}\coloneqq H_{\mathbb{R}}+\mathrm{i}H_{\mathbb{R}}$.
Denote by $H$ the completion of $H_{\mathbb{C}}$ with respect to
$\langle\cdot,\cdot\rangle_{U}$ and denote by $\mathcal{F}_{q}(H)$
the $q$-Fock space associated to $H$. We define  the left and right
creation operators
\[
l(\xi)\eta=\xi\otimes\eta,\quad r(\xi)\eta=\eta\otimes\xi,\quad\xi\in H,\eta\in\mathcal{F}_{q}(H)
\]
and the left and right annihilation operators
\[
l^{*}(\xi)=l(\xi)^{*},\quad r^{*}(\xi)=r(\xi)^{*},\quad\xi\in H.
\]
We denote by $\Gamma_{q}(H_{\mathbb{R}},U_{t})$ (resp. $C_{q}^{*}(H_{\mathbb{R}},U_{t})$)
the von Neumann algebra (resp. C{*}-algebra) generated
by $\{l(e)+l^{*}(e):e\in H_{\mathbb{R}}\}$ in $B(\mathcal{F}_{q}(H))$,
to be called the $q$-Araki-Woods von Neumann algebra. Properties of the vacuum state guarantee the existence of the
\emph{Wick product} map $W:\Gamma_{q}(H_{\mathbb{R}},U_{t})\Omega\to\Gamma_{q}(H_{\mathbb{R}},U_{t})$
such that $W(\xi)\Omega=\xi$. On the other hand, denote
\[
H_{\mathbb{R}}'=\{\xi\in H:\ \forall\eta\in H_{\mathbb{R}},\ \langle\xi,\eta\rangle\in\mathbb{R}\}.
\]
Then the von Neumann algebra $\Gamma_{q,r}(H_{\mathbb{R}},U_{t})$
generated by $\{r(e)+r^{*}(e):e\in H_{\mathbb{R}}'\}$ in $B(\mathcal{F}_{q}(H))$
is the commutant of $\Gamma_{q}(H_{\mathbb{R}},U_{t})$, and again there
exist a right Wick product\emph{ }$W_{r}:\Gamma_{q,r}(H_{\mathbb{R}},U_{t})\Omega\to\Gamma_{q,r}(H_{\mathbb{R}},U_{t})$
such that $W_{r}(\xi)\Omega=\xi$. We denote by $I$ the standard complex conjugation
on $H_{\mathbb{R}}+\mathrm{i}H_{\mathbb{R}}$, and by $I_{r}$ the complex
conjugation on $H_{\mathbb{R}}'+\mathrm{i}H_{\mathbb{R}}'$. The following
observations are well-known and we state them here for later use.
\begin{lem}
\label{lem:qaw preliminary lem}\emph{(1)} Suppose that $e_{1},\dots,e_{n}\in H_{\mathbb{C}}$.
Then we have the following Wick formula
\begin{equation}
W(e_{1}\otimes\dots\otimes e_{n})=\sum_{k=0}^{n}\sum_{i_{1},\dots,i_{k},j_{k+1},\dots,j_{n}}l(e_{i_{1}})\dots l(e_{i_{k}})l^{*}(Ie_{j_{k+1}})\dots l^{*}(Ie_{j_{n}})q^{i(I_{1},I_{2})},\label{Wickformula}
\end{equation}
where $I_{1}=\{i_{1},\dots,i_{k}\}$ and $I_{2}=\{j_{k+1},\dots,j_{n}\}$
form a partition of the set $\{1,\dots,n\}$ and $i(I_{1},I_{2})$
is the number of crossings. A similar formula holds for $W_{r}(e_{1}\otimes\cdots\otimes e_{n})$
as well.

\emph{(2)} Let $f\in H_{\mathbb{R}}$, $e\in H_{\mathbb{R}}'+\mathrm{i}H_{\mathbb{R}}'$.
If $\langle e,f\rangle=0$, then $\langle I_{r}e,f\rangle=0$.\end{lem}
\begin{proof}
(1) See \cite[Proposition 2.7]{bozejkokummerspeicher97q}, \cite[Lemma 3.1]{wasilewski16qhaargerup}.

(2) Write $e=e_{1}+\mathrm{i}e_{2}$ with $e_{1},e_{2}\in H_{\mathbb{R}}'$.
Since $\langle e_{1},f\rangle\in\mathbb{R},\langle e_{2},f\rangle\in\mathbb{R}$,
we see that the identity $\langle e,f\rangle=0$ yields
\[
\langle e_{1},f\rangle=\langle e_{2},f\rangle=0.
\]
Therefore
\[
\langle I_{r}e,f\rangle=\langle e_{1}-\mathrm{i}e_{2},f\rangle=0.
\]

\end{proof}
According to Shlyakhtenko \cite{shlyakhtenko97freeaw}, we have the
decomposition
\[
(H_{\mathbb{R}},U_{t})=(K_{\mathbb{R}},U_{t}')\oplus(L_{\mathbb{R}},U_{t}''),
\]
where $U_{t}'$ is almost periodic and $U_{t}''$ is ergodic. Then
$K_{\mathbb{R}}\subset H_{\mathbb{R}}$ is the real closed subspace
spanned by eigenvectors of $U_{t}=A^{\mathrm{i}t}$. Let $K_{\mathbb{C}}=K_{\mathbb{R}}+\mathrm{i}K_{\mathbb{R}}$
be the complexification and $K$ be the completion of $K_{\mathbb{C}}$ with respect to the
deformed norm as above, and similarly for $L$. Note that the orthogonal
projection $P:H_{\mathbb{R}}\to K_{\mathbb{R}}$ commutes with $U_{t}$.
So by the second quantization, $\Gamma_{q}(K_{\mathbb{R}},U_{t}|_{K})$
embeds as a von Neumann subalgebra of $\Gamma_{q}(H_{\mathbb{R}},U_{t})$.
For an operator $T$ we denote by $\mathcal{F}_{q}(T)$ its second
quantization.

The following observation shows that in looking at the centre of the q-Araki-Woods algebra it suffices to consider the `$K$-part' of the algebra (we do not really use this fact in the sequel).

\begin{lem}
\emph{(1) }The semigroup $\mathcal{F}_{q}(U_{t})$ admits no eigenvectors
in $\mathcal{F}_{q}(K)^{\bot}\subset\mathcal{F}_{q}(H)$;

\emph{(2)} Assume $x\in\Gamma_{q}(H_{\mathbb{R}},U_{t})\cap\Gamma_{q}(H_{\mathbb{R}},U_{t})'$.
Then $x\Omega\in\mathcal{F}_{q}(K)$ and $x\in\Gamma_{q}(K_{\mathbb{R}},U_{t}|_{K_{\mathbb{R}}})$.\end{lem}
\begin{proof}
(1) Let $(e_{i})$ be an orthonormal basis in $H_{\mathbb{R}}$. Since
$\mathcal{F}_{q}(P)$ is the orthogonal projection onto $\mathcal{F}_{q}(K)$,
we have
\[
\mathcal{F}_{q}(P)(\mathcal{F}_{q}(K)^{\bot})=0.
\]
Hence
\[
\mathcal{F}_{q}(K)^{\bot}=\overline{\mathrm{span}}\{e_{i_{1}}\otimes\cdots\otimes e_{i_{n}}:n\geq1,\exists1\leq m\leq n,e_{i_{m}}\in L_{\mathbb{R}}\}.
\]
Denote
\[
K_{n}=\overline{\mathrm{span}}\{e_{i_{1}}\otimes\cdots\otimes e_{i_{n}}\in\mathcal{F}_{q}(K)^{\bot}\}=\mathrm{span}\{H_{i_{1}}\otimes\cdots\otimes H_{i_{n}},H_{i}=K\text{ or }L,\exists H_{i}=L\}.
\]
Note that $U_{t}$ is unitarily equivalent to a multiplier map on
some $L^{2}(\mu)$. So by the definition of $K$ and $L$ and the fact that at least one of $H_{i_k}$ is equal to $L$, it is easy to see that
$\mathcal{F}_{q}(U_{t})|_{H_{i_{1}}\otimes\cdots\otimes H_{i_{n}}}=U_{t}|_{H_{i_{1}}}\otimes\cdots\otimes U_{t}|_{H_{i_{n}}}$
admits no eigenvectors. Since each $H_{i_{n}}$ is invariant under
$U_{t}$, $\mathcal{F}_{q}(U_{t})$ admits no eigenvectors in $K_{n}$
either. Then the lemma follows immediately. Indeed, let
\[
\xi=\sum_{n}\xi_{n}\in\mathcal{F}_{q}(K)^{\bot},\quad\xi_{n}\in K_{n},
\]
be an eigenvector. Then we get
\[
\sum_{n}(U_{t}\xi_{n}-\lambda\xi_{n})=0,
\]
for some $\lambda$ and hence $U_{t}\xi_{n}-\lambda\xi_{n}=0$ for
all $n$, which yields a contradiction.

(2) Assume $x\in\Gamma_{q}(H_{\mathbb{R}},U_{t})\cap\Gamma_{q}(H_{\mathbb{R}},U_{t})'$.
Note that $x$ is in the centralizer of the vacuum state $\varphi$.
So we have for all $t\in\mathbb{R}$,
\[
\sigma_{t}(x)\Omega=\Delta^{\mathrm{i}t}x\Delta^{-\mathrm{i}t}\Omega=x\Omega.
\]
Recall the Tomita-Takesaki theory for $\Gamma_{q}(H_{\mathbb{R}},U_{t})$ and the vacuum state.
We see that $x\Omega$ is a fixed point of $\mathcal{F}_{q}(U_{t})$,
and hence $(\mathcal{F}_{q}(P)^{\bot})(x\Omega)$ is an eigenvector
by orthogonal decomposition. So by the above lemma $(\mathcal{F}_{q}(P)^{\bot})(x\Omega)=0$.
That is, $x\Omega\in\mathcal{F}_{q}(K)$ and $x\in\Gamma_{q}(K_{\mathbb{R}},U_{t}|_{K})$.
\end{proof}

\begin{prop}
\label{prop:qaw prop}Let $D_{\mathbb{R}}\subset H_{\mathbb{R}}$
be a real finite-dimensional Hilbert subspace and let $M$ be a diffuse
abelian von Neumann subalgebra of $\Gamma_{q}(H_{\mathbb{R}},U_{t})$
such that $M\Omega\subset\mathcal{F}_{q}(D)$, where $D=D_{\mathbb{R}}+ i D_{\mathbb{R}}$. Assume $x\in\Gamma_{q}(H_{\mathbb{R}},U_{t})\cap M'$.

\emph{(1)} If $x\in C_{q}^{*}(H_{\mathbb{R}},U_{t})$, then $x\Omega\in\mathcal{F}_{q}(D)$;

\emph{(2)} If $M$ is  contained in the centralizer of $\Gamma_{q}(H_{\mathbb{R}},U_{t})$ ,
then $x\Omega\in\mathcal{F}_{q}(D)$.\end{prop}
\begin{proof}
The proof is similar to that of Theorem \ref{thm:factor mixed}, so we
only present a sketch. Since $M$ is diffuse and $M\Omega\subset\mathcal{F}_{q}(D)$,
we may find a sequence of unitaries $(u_{\alpha})_{\alpha\in \mathbb{N}}\subset M$
such that
\[
u_{\alpha}=u_{\alpha}^{*},\quad u_{\alpha}^{2}=1,\quad u_{\alpha}\Omega\to0\text{ weakly in }\mathcal{F}_{q}(D).
\]
We may show that for any vector $\xi\in H^{\otimes n}$ with $n\geq1$
which is orthogonal to $\mathcal{F}_{q}(D)$, and for $w\in\Gamma_{q}(H_{\mathbb{R}},U_{t})$,
if one of the following conditions is satisfied:

(a) $w\in C_{q}^{*}(H_{\mathbb{R}},U_{t})$;

(b) the operator $z\Omega\mapsto zu_{\alpha}\Omega$ is uniformly
bounded on $\mathcal{F}_{q}(H)$; \\
then
\begin{equation}
\varphi(u_{\alpha}w^{*}u_{\alpha}W(\xi))=\langle W_{r}(\xi)u_{\alpha}\Omega,wu_{\alpha}\Omega\rangle\to0.\label{eq:conv qaw}
\end{equation}
Indeed, we note that the anti-linear functional $z\mapsto\varphi(u_{\alpha}z^{*}u_{\alpha}W(\xi))$
is uniformly bounded on $C_{q}^{*}(H_{\mathbb{R}},U_{t})$ with respect
to $\alpha$, and if (b) is satisfied, the anti-linear functional $z\Omega\mapsto\varphi(u_{\alpha}z^{*}u_{\alpha}W(\xi))$
is uniformly bounded on $\mathcal{F}_{q}(H)$ with respect to $\alpha$.
So if any one of (a) and (b) is satisfied, we may find a sequence
of vectors $(\eta_{k})_{k=1}^{\infty}$ in the algebraic span of $\{H^{\otimes n}:n\geq1\}$
such that we have the convergence
\[
\varphi(u_{\alpha}W(\eta_{k})^{*}u_{\alpha}W(\xi))\to\varphi(u_{\alpha}w^{*}u_{\alpha}W(\xi)),\quad k\to\infty
\]
which is uniform with respect to $\alpha$. This means that in order
to see \eqref{eq:conv qaw} under the condition (a) or (b), it suffices
to assume that $w$ belongs to the the algebraic span of $\{H^{\otimes n}:n\geq1\}$.
On the other hand, recall that $\xi\bot\mathcal{F}_{q}(D)$, which
means that $\xi$ is the combination of words of the form
\[
e_{m_{1}}\otimes\cdots\otimes e_{m_{n}},\quad e_{m_{1}},\ldots,e_{m_{n}}\in H\cup D^{\bot},\exists1\leq k\leq n,\ e_{m_{k}}\in D^{\bot}.
\]
Thus by the Wick formula in Lemma \ref{Wickformula}, it suffices
to prove the convergence
\[
\langle r(e_{i_{1}})\cdots r(e_{i_{m}})r^{*}(I_{r}e_{i_{m+1}})\cdots r^{*}(I_{r}e_{i_{n}})u_{\alpha}\Omega,l(e_{j_{1}})\cdots l(e_{j_{s}})l^{*}(Ie_{j_{s+1}})\cdots l^{*}(Ie_{j_{p}})u_{\alpha}\Omega\rangle\to0,
\]
where there is $1\leq k\leq n$ such that $e_{i_{k}}\in D^{\bot}$,
$e_{i_{k'}}\in H$ for $1\leq k<k'$. By Lemma \ref{lem:qaw preliminary lem},
$I_{r}e_{i_{k}}\in D^{\bot}$ holds as well. Consequently, if $k\geq m+1$,
then $r^{*}(I_{r}e_{i_{k}})\cdots r^{*}(I_{r}e_{i_{n}})u_{\alpha}\Omega=0$
and the above convergence is trivial. Hence we assume $k\leq m$.
Recall that
\[
l^{*}(f)r^{*}(g)-r^{*}(g)l^{*}(f)=0,\quad l(f)r^{*}(g)-r^{*}(g)l(f)=\langle f,g\rangle q^{k}\left(\oplus_{k\geq0}\mathrm{id}_{H^{\otimes k}}\right),\quad f,g\in H.
\]
Now applying Lemma \ref{lem:conv general case} as in Theorem \ref{thm:factor mixed},
we obtain the desired convergence \eqref{eq:conv qaw}.

Now the conclusion of the theorem is immediate. Take $x\in\Gamma_{q}(H_{\mathbb{R}},U_{t})\cap M'$.
We have for all $\alpha\geq1$ and every $\xi\in H^{\otimes n}$ with
$n\geq1$ which is orthogonal to $\mathcal{F}_{q}(D)$,
\[
\langle\xi,x\Omega\rangle=\varphi(x^{*}W(\xi))=\varphi(x^{*}u_{\alpha}^{2}W(\xi))=\varphi(u_{\alpha}x^{*}u_{\alpha}W(\xi))=\langle W_{r}(\xi)u_{\alpha}\Omega,xu_{\alpha}\Omega\rangle.
\]
If now the assumption of (1) holds, then by (a) and \eqref{eq:conv qaw} we see that
\[
\langle\xi,x\Omega\rangle=\langle W_{r}(\xi)u_{\alpha}\Omega,xu_{\alpha}\Omega\rangle\to0.
\]
Similarly if the assumption of (2) holds, then the $u_{\alpha}$'s belong to the centralizer of
$\Gamma_{q}(H_{\mathbb{R}},U_{t})$, and hence
\[
\|zu_{\alpha}\Omega\|^{2}=\varphi(u_{\alpha}z^{*}zu_{\alpha})=\varphi(z^{*}zu_{\alpha}^{2})=\varphi(z^{*}z)=\|z\Omega\|^{2},
\]
so (b) is satisfied. By \eqref{eq:conv qaw} this yields that
\[
\langle\xi,x\Omega\rangle=\langle W_{r}(\xi)u_{\alpha}\Omega,xu_{\alpha}\Omega\rangle\to0.
\]
So $\langle\xi,x\Omega\rangle=0$ for all words $\xi\in\mathcal{F}_{q}(D)^{\bot}$
and hence $x\Omega\in\mathcal{F}_{q}(D)$.\end{proof}

We are ready to state the second main result of this article.

\begin{thm}
Assume $\dim H_{\mathbb{R}}\geq2$.

\emph{(1)} If there exists $\xi_{0}\in H_{\mathbb{R}}$ such that
$U_{t}\xi_{0}=\xi_{0}$, then $\Gamma_{q}(H_{\mathbb{R}},U_{t})$
is a factor;

\emph{(2)} Let $H_{\mathbb{R}}^{(1)},H_{\mathbb{R}}^{(2)}$ be two
finite-dimensional Hilbert subspaces of \textup{$H_{\mathbb{R}}$
}\textup{\emph{which are invariant under $U_{t}$ and are orthogonal
with respect to the real inner product of }}\textup{$H_{\mathbb{R}}$}\textup{\emph{.
Assume that for $k=1,2$ the centralizer of $\Gamma_{q}(H_{\mathbb{R}}^{(k)},U_{t}|_{H_{\mathbb{R}}^{(k)}})$
contains a diffuse element. Then}}\textup{ }$\Gamma_{q}(H_{\mathbb{R}},U_{t})$
is a factor;

\emph{(3)} $\Gamma_{q}(H_{\mathbb{R}},U_{t})'\cap C_{q}^{*}(H_{\mathbb{R}},U_{t})=\mathbb{C}1$.\end{thm}
\begin{proof}
(1) Since $\dim H_{\mathbb{R}}\geq2$ and $U_{t}\xi_{0}=\xi_{0}$,
the subspace $(\mathbb{C}\xi_{0})^{\bot}\subset H$ is invariant
under $U_{t}$, and we may find a vector $\eta\in(\mathbb{C}\xi_{0})^{\bot}$
such that $\eta\in H_{\mathbb{R}}'$, $\eta\bot\xi_{0}$. Note that
in this case $W_{r}(\eta)=W_{r}(\eta)^{*}$ and $I\eta\bot\xi_{0}$.
Take $x\in\Gamma_{q}(H_{\mathbb{R}},U_{t})'\cap\Gamma_{q}(H_{\mathbb{R}},U_{t})$
and denote $\xi=x\Omega$. Note that $W(\xi_{0})$ belongs to the centralizer
of $\Gamma_{q}(H_{\mathbb{R}},U_{t})$ by the assumption $U_{t}\xi_{0}=\xi_{0}$,
and that the spectral measure of $W(\xi_{0})$ is $q$-semicircular
(\cite[Remarks p.298-299]{nou06qawqwep}) and hence $W(\xi_{0})$
generates a diffuse abelian von Neumann subalgebra. So by Proposition
\ref{prop:qaw prop}(2), we have
\[
\xi\in\mathcal{F}_{q}(\mathbb{C}\xi_{0}),\quad\eta\bot\xi,I\eta\bot\xi.
\]
Then we see that
\begin{align*}
W(\xi)\eta & =xW(\eta)\Omega=W(\eta)x\Omega=W(\eta)\xi\\
 & =l(\eta)\xi+l^{*}(I\eta)\xi=\eta\otimes\xi.
\end{align*}
As a result, writing
\[
\lambda=\langle\xi,\Omega\rangle,\quad\zeta=\xi-\lambda\Omega,
\]
we have
\begin{align*}
\|\eta\otimes\xi\|^{2} & =\langle\eta\otimes\xi,W(\xi)\eta\rangle=\langle\eta\otimes\xi,W_{r}(\eta)\xi\rangle=\langle W_{r}(\eta)(\eta\otimes\xi),\xi\rangle\\
 & =\lambda\langle W_{r}(\eta)\eta,\xi\rangle+\langle W_{r}(\eta)(\eta\otimes\zeta),\xi\rangle\\
 & =\lambda\langle\|\eta\|^{2}\Omega,\xi\rangle+\lambda\langle\eta\otimes\eta,\xi\rangle+\langle\eta\otimes\zeta\otimes\eta,\xi\rangle\\
 & =|\lambda|^{2}\|\eta\|^{2}
\end{align*}
where we have used the relation $\eta\bot\xi_0$ in the last equality.
However
\[
\|\eta\otimes\xi\|^{2}=\|\eta\otimes(\lambda\Omega+\zeta)\|^{2}=|\lambda|^{2}\|\eta\|^{2}+\|\eta\otimes\zeta\|^{2}.
\]
Thus the above two equalities yield that $\eta\otimes\zeta=0$. Therefore
$\zeta=0$ and $x\Omega=\xi=\lambda\Omega$. This proves that
\[
\Gamma_{q}(H_{\mathbb{R}},U_{t})'\cap\Gamma_{q}(H_{\mathbb{R}},U_{t})=\mathbb{C}1.
\]

(2) This assertion follows directly from Proposition \ref{prop:qaw prop}(2)
since according to that result any $x\in\Gamma_{q}(H_{\mathbb{R}},U_{t})'\cap\Gamma_{q}(H_{\mathbb{R}},U_{t})$
should satisfy
\[
x\Omega\in\mathcal{F}_{q}(H^{(1)})\cap\mathcal{F}_{q}(H^{(2)})(=\mathbb{C}\Omega).
\]

(3) Since $\dim H_{\mathbb{R}}\geq2$, we may find two vectors $e_{1},e_{2}\in H_{\mathbb{R}}$
which are orthogonal with respect to the real inner product of $H_{\mathbb{R}}$.
Then $W(e_{1})$ and $W(e_{2})$ are self-adjoint diffuse elements
as discussed before, and $\mathcal{F}_{q}(\mathbb{C}e_{1})\cap\mathcal{F}_{q}(\mathbb{C}e_{2})=\mathbb{C}\Omega$.
Then according to Proposition \ref{prop:qaw prop}(1), any $x\in\Gamma_{q}(H_{\mathbb{R}},U_{t})'\cap C_{q}^{*}(H_{\mathbb{R}},U_{t})$
should satisfy
\[
x\Omega\in\mathcal{F}_{q}(\mathbb{C}e_{1})\cap\mathcal{F}_{q}(\mathbb{C}e_{2})(=\mathbb{C}\Omega).
\]
Therefore the assertion is proved.
\end{proof}
\subsection*{Acknowledgment}
The authors would like to thank \'Eric Ricard and Mateusz Wasilewski for helpful discussions, and the anonymous referee for careful reading of our manuscript. The authors were partially supported by the NCN (National Centre of Science) grant
2014/14/E/ST1/00525.

\end{document}